\documentclass[11pt]{article}
\usepackage{amsmath,amssymb,amsthm}
\usepackage{geometry}
\usepackage{comment}
\usepackage{tikz}
\usepackage{amsthm}
\usepackage[colorlinks=true, linkcolor=blue]{hyperref}

\newtheorem{theorem}{Theorem}
\newtheorem{proposition}{Proposition}
\newtheorem{remark}{Remark}
\newtheorem{lemma}{Lemma}

\usepackage[english]{babel}  
\usepackage{appendix}        
\usepackage{listings}
\usepackage{xcolor}

\title{A 910-block explicit construction guaranteeing a triple intersection\\
with every $6$-subset of $[60]$}
\author{Paulo Henrique Cunha Gomes\\
School of Technology (FT), University of Campinas (Unicamp)\\
Limeira, 13484-332, S\~ao Paulo, Brazil\\
\texttt{p072049@dac.unicamp.br}}
\date{}

\begin{document}
\maketitle

\begin{abstract}
We present a simple explicit family $\mathcal{B}$ of $910$ $6$-subsets of $[60]=\{1,\dots,60\}$
such that every $6$-subset $S\subset[60]$ intersects at least one block $B\in\mathcal{B}$ in at least
three elements, i.e.\ $|S\cap B|\ge 3$.
Equivalently, $\mathcal{B}$ is a covering (dominating set) of the Johnson graph $J(60,6)$ with covering
radius $3$ in the Johnson metric.
The construction is purely combinatorial, based on a fixed split of $[60]$ into two halves, a pairing
of each half, and a pigeonhole argument.
We also record a crude counting lower bound and a straightforward generalization to $[2m]$ (with $m$ even).
\end{abstract}

\medskip
\noindent\textbf{Keywords:} covering designs; constant-weight covering codes; set systems; Johnson graph; explicit construction; pigeonhole principle.\\
\textbf{2020 MSC:} 05B40; 94B05; 05C69.

\section{Introduction}

Let $[n]=\{1,2,\dots,n\}$ and write $\binom{[n]}{k}$ for the family of $k$-subsets of $[n]$.
Fix $n=60$ and $k=6$.
We study families $\mathcal{B}\subset \binom{[60]}{6}$ with the \emph{triple-intersection property}
\begin{equation}\label{eq:triple}
\forall S\in \binom{[60]}{6}\ \exists B\in\mathcal{B}\ \text{such that}\ |S\cap B|\ge 3.
\end{equation}

This property can be phrased in the language of Johnson graphs and constant-weight covering codes.
The Johnson graph $J(60,6)$ has vertex set $\binom{[60]}{6}$, where two vertices are adjacent
if they differ in exactly one element.
The Johnson distance between $S,B\in\binom{[60]}{6}$ is
\begin{equation}\label{eq:johnson-distance}
d_J(S,B)=6-|S\cap B|.
\end{equation}
Thus \eqref{eq:triple} is equivalent to requiring that every vertex $S$ lies within
distance at most $3$ from some block $B\in\mathcal{B}$, i.e.\ $\mathcal{B}$ is a covering (dominating set)
of $J(60,6)$ with covering radius $3$.

The goal of this short note is to give a clean, reproducible benchmark construction for \eqref{eq:triple}
with an explicit and elementary proof.

Covering problems of this type (in Johnson graphs, covering designs, and constant-weight covering codes)
have been studied extensively; see, for example, \cite{ColbournDinitz,CohenEtAl,Brouwer}.
In this note we do \emph{not} claim that the size $|\mathcal{B}|=910$ is optimal.What constructions are possible to reduce this number?
\[
6 \cdot \binom{9}{3} = 892; \quad 2 \cdot \binom{12}{3} = 828; \quad \binom{15}{3} + 388 = 843
\]
Rather, our aim is to provide a transparent explicit construction giving a concrete upper bound,
together with a simple (and admittedly crude) counting lower bound recorded in Section~4.

\section{Construction (910 blocks)}
\label{sec:construction}

We begin by fixing an explicit partition of the ground set $[60]$ into ten pairwise disjoint
\emph{base blocks} of size $6$:
\[
G_i=\{6(i-1)+1,\,6(i-1)+2,\,\dots,\,6i\}\qquad (1\le i\le 10).
\]
Thus $G_i\cap G_j=\emptyset$ for $i\neq j$ and $\bigcup_{i=1}^{10}G_i=[60]$.

Next, split these ten base blocks into two groups of five,
\[
\mathcal{G}_1=\{G_1,\dots,G_5\},\qquad \mathcal{G}_2=\{G_6,\dots,G_{10}\},
\]
and define the induced bipartition of $[60]$ by
\begin{equation}\label{eq:halves}
U_1=\bigcup_{G\in\mathcal{G}_1}G=\{1,\dots,30\},\qquad
U_2=\bigcup_{G\in\mathcal{G}_2}G=\{31,\dots,60\}.
\end{equation}
Hence $U_1\cap U_2=\emptyset$, $U_1\cup U_2=[60]$, and $|U_1|=|U_2|=30$.

Partition each part $U_t$ into $15$ disjoint pairs:
\begin{align*}
P_i&=\{2i-1,2i\}\subset U_1\qquad (1\le i\le 15),\\
Q_i&=\{30+2i-1,30+2i\}\subset U_2\qquad (1\le i\le 15).
\end{align*}
Let
\[
\mathcal{P}_1=\{P_1,\dots,P_{15}\},\qquad
\mathcal{P}_2=\{Q_1,\dots,Q_{15}\}.
\]
Equivalently, each base block $G_i$ is internally partitioned into the three disjoint pairs
$\{6(i-1)+1,6(i-1)+2\}$, $\{6(i-1)+3,6(i-1)+4\}$, and $\{6(i-1)+5,6(i-1)+6\}$.

Define two families of $6$-subsets by taking unions of three pairs inside a part:
\begin{align}
\mathcal{B}_1&=\{P_i\cup P_j\cup P_k:\ 1\le i<j<k\le 15\},\label{eq:B1}\\
\mathcal{B}_2&=\{Q_i\cup Q_j\cup Q_k:\ 1\le i<j<k\le 15\}.\label{eq:B2}
\end{align}
Each block in $\mathcal{B}_1$ and $\mathcal{B}_2$ has size $6$, and
\begin{equation}\label{eq:size}
|\mathcal{B}_1|=|\mathcal{B}_2|=\binom{15}{3}=455.
\end{equation}
Finally, set
\begin{equation}\label{eq:B}
\mathcal{B}=\mathcal{B}_1\cup \mathcal{B}_2.
\end{equation}
Then $|\mathcal{B}|=455+455=910$.

\section{Main result}

\begin{theorem}\label{thm:main}
Let $\mathcal{B}$ be the family of $910$ blocks defined in \eqref{eq:B}.
For every $S\subset[60]$ with $|S|=6$, there exists $B\in\mathcal{B}$ such that $|S\cap B|\ge 3$.
Equivalently, $\mathcal{B}$ has covering radius $3$ in $J(60,6)$ with respect to $d_J$ in \eqref{eq:johnson-distance}.
\end{theorem}
\begin{proof}
Let $S\subset[60]$ be any $6$-subset. Using the bipartition \eqref{eq:halves}, set
\[
S_t=S\cap U_t\qquad (t\in\{1,2\}).
\]
Since $U_1\cup U_2=[60]$ is a disjoint union, we have $|S_1|+|S_2|=6$, hence
$\max\{|S_1|,|S_2|\}\ge 3$. Fix $t\in\{1,2\}$ such that $|S_t|\ge 3$.

Choose distinct elements $a,b,c\in S_t$. Because $\mathcal{P}_t$ is a partition of $U_t$ into disjoint pairs,
each of $a,b,c$ lies in a unique pair of $\mathcal{P}_t$; write
\[
\pi_t(x)\in\mathcal{P}_t\quad\text{for the unique pair containing }x\in U_t.
\]
Then the set of pairs $\{\pi_t(a),\pi_t(b),\pi_t(c)\}$ has cardinality $2$ or $3$.

If $|\{\pi_t(a),\pi_t(b),\pi_t(c)\}|=3$, define
\[
B=\pi_t(a)\cup \pi_t(b)\cup \pi_t(c)\in\mathcal{B}_t\subset\mathcal{B}.
\]
If $|\{\pi_t(a),\pi_t(b),\pi_t(c)\}|=2$, then two of $a,b,c$ lie in the same pair.
Let $D=\{\pi_t(a),\pi_t(b),\pi_t(c)\}$; in this case $|D|=2$.
Since $|\mathcal{P}_t|=15$, at least $15-2=13$ choices remain for a third pair $R\in\mathcal{P}_t\setminus D$.
Set
\[
B=\pi_t(a)\cup \pi_t(b)\cup R\in\mathcal{B}_t\subset\mathcal{B}.
\]
In both cases, $a,b,c\in B$, so $|S\cap B|\ge 3$.

Finally, by \eqref{eq:johnson-distance} we have $d_J(S,B)=6-|S\cap B|$, so $|S\cap B|\ge 3$
is equivalent to $d_J(S,B)\le 3$. Hence $\mathcal{B}$ has covering radius $3$ in $J(60,6)$.
\end{proof}

\subsection*{Worked example}
Let
\[
S=\{1,2,7,9,31,45\}.
\]
Then $|S\cap U_1|=4$, so we work inside $U_1$.
We have $1,2\in P_1=\{1,2\}$, $7\in P_4=\{7,8\}$, and $9\in P_5=\{9,10\}$.
Choosing the three indices $\{1,4,5\}$ gives the block
\[
B=P_1\cup P_4\cup P_5=\{1,2,7,8,9,10\}\in\mathcal{B}_1,
\]
and indeed $|S\cap B|=4\ge 3$.

\section{Lower bounds on the number of blocks}\label{sec:lower}

Let $M$ denote the minimum possible size of a family $\mathcal{F}\subset\binom{[60]}{6}$
satisfying the triple-intersection requirement
\[
\forall S\in\binom{[60]}{6}\ \exists B\in\mathcal{F}\ \text{such that}\ |S\cap B|\ge 3.
\]
Equivalently, $\mathcal{F}$ is a covering (dominating set) of the Johnson graph $J(60,6)$
with covering radius $3$ in the Johnson metric $d_J(S,B)=6-|S\cap B|$.

\subsection*{Sphere-covering (volume) bound}

Fix a block $B\in\binom{[60]}{6}$ and define its radius-$3$ neighborhood
\[
\mathcal{N}_3(B)=\{S\in\binom{[60]}{6}:\ d_J(S,B)\le 3\}
              =\{S\in\binom{[60]}{6}:\ |S\cap B|\ge 3\}.
\]
The cardinality $|\mathcal{N}_3(B)|$ depends only on the parameters, not on $B$.
Indeed, for $i=|S\cap B|$ we choose $i$ elements from $B$ and $6-i$ from the complement
$[60]\setminus B$ (which has size $54$). Hence
\begin{equation}\label{eq:N}
|\mathcal{N}_3(B)|
=\sum_{i=3}^{6}\binom{6}{i}\binom{54}{6-i}
=517{,}870.
\end{equation}

Now let $\mathcal{F}$ be any family with $|\mathcal{F}|=M$ satisfying the requirement.
Then the neighborhoods $\{\mathcal{N}_3(B):B\in\mathcal{F}\}$ cover all vertices
$\binom{[60]}{6}$, so by a union bound,
\[
\binom{60}{6}
=\left|\binom{[60]}{6}\right|
\le \sum_{B\in\mathcal{F}}|\mathcal{N}_3(B)|
= M\cdot 517{,}870.
\]
Therefore,
\begin{equation}\label{eq:lower}
M\ \ge\ \left\lceil\frac{\binom{60}{6}}{517{,}870}\right\rceil
=\left\lceil 96.67\ldots\right\rceil
=97.
\end{equation}

\begin{remark}[What the bound does \emph{and does not} say]
Inequality \eqref{eq:lower} is a \emph{necessary} condition:
with fewer than $97$ blocks, no construction can guarantee a triple intersection
with every $6$-subset.
However, \eqref{eq:lower} does \emph{not} imply that $97$ blocks are sufficient,
because the neighborhoods $\mathcal{N}_3(B)$ typically overlap heavily.
Thus the exact optimum $M$ may be strictly larger than $97$.
\end{remark}

Combining \eqref{eq:lower} with Theorem~\ref{thm:main} yields the current gap
\[
97 \le M \le 910.
\]
Determining $M$ exactly, or substantially improving either bound, is left open.

\section{Remarks and a generalization}

\begin{remark}[Why the two-halves split works]
The proof of Theorem~\ref{thm:main} uses only the fixed bipartition $[60]=U_1\sqcup U_2$:
every $6$-set must place at least three points in one half, and the internal pairing of that half
then forces a block containing those three points.
By contrast, if one partitions $[60]$ into more than two groups, a $6$-set can distribute itself
as $(2,2,2)$ across three groups, so no single group contains a triple; in that setting,
arguments that recombine pairs \emph{within one group} need not force a triple intersection.
This observation is purely explanatory and not needed in the proof above.
\end{remark}

\begin{proposition}[Generalization to $\lbrack 2m\rbrack$, $m$ even]\label{prop:general}

Let $m$ be even and set $n=2m$.
Split $[n]$ into halves $U_1$ and $U_2$ of size $m$ and partition each half into $m/2$ disjoint pairs.
Let $\mathcal{B}$ be the union of the two families formed by taking unions of three pairs inside one half.
Then $|\mathcal{B}|=2\binom{m/2}{3}$ and $\mathcal{B}$ satisfies \eqref{eq:triple} for $k=6$ on $[n]$.
\end{proposition}

\begin{proof}
Identical to the proof of Theorem~\ref{thm:main}: any $6$-subset has at least three elements in one half,
and those three elements lie in at most three designated pairs in that half; a union of those pairs
(and, if needed, one extra pair) yields a block meeting the $6$-subset in at least three points.
\end{proof}
\section{Integer solutions of $x_1+x_2+x_3=6$}
\label{sec:solutions-x123}

Let $x_1,x_2,x_3\in\mathbb{Z}_{\ge 0}$ (zeros allowed). The number of nonnegative integer
solutions to
\[
x_1+x_2+x_3=6
\]
 It represents the number of ways to distribute the 6 elements of 
S among the three groups of blocks, is given by the stars-and-bars formula:
\[
\binom{6+3-1}{3-1}=\binom{8}{2}=28.
\]

For completeness, we list all $28$ solutions in increasing order of $x_1$:

\begin{itemize}
  \item $x_1=0$: $(0,0,6),(0,1,5),(0,2,4),(0,3,3),(0,4,2),(0,5,1),(0,6,0)$.
  \item $x_1=1$: $(1,0,5),(1,1,4),(1,2,3),(1,3,2),(1,4,1),(1,5,0)$.
  \item $x_1=2$: $(2,0,4),(2,1,3),\textbf{(2,2,2)},(2,3,1),(2,4,0)$.
  \item $x_1=3$: $(3,0,3),(3,1,2),(3,2,1),(3,3,0)$.
  \item $x_1=4$: $(4,0,2),(4,1,1),(4,2,0)$.
  \item $x_1=5$: $(5,0,1),(5,1,0)$.
  \item $x_1=6$: $(6,0,0)$.
\end{itemize}

\paragraph{Why a $(3-blocks,3-blocks,4-blocks)$ split of the ten base blocks is not sufficient.}
One might hope that partitioning the ten base blocks into three groups,
say $\mathcal{G}_1$ with $3$ blocks, $\mathcal{G}_2$ with $3$ blocks, and $\mathcal{G}_3$ with $4$ blocks,
could lead to an even more economical construction, by recombining the three disjoint pairs
\emph{within each group} to generate candidate $6$-blocks.
However, this strategy fails in general because of the balanced distribution pattern $(2,2,2)$:
there exist $6$-subsets $S\subset[60]$ such that $|S\cap U_1|=|S\cap U_2|=|S\cap U_3|=2$,
where $U_t=\bigcup_{B\in\mathcal{G}_t}B$ is the union of the elements covered by group $\mathcal{G}_t$.
For such an $S$, any recombined block built solely from pairs inside a single group $\mathcal{G}_t$
is contained in $U_t$ and therefore intersects $S$ in at most $2$ points, i.e., $|S\cap B|\le 2$.
Consequently, groupwise pair-recombination under a $(3-blocks,3-blocks,4-blocks)$ partition does \emph{not} guarantee
the appearance of a triple intersection after recombination.

After recombining the pairs within each group, we would obtain
\[
\binom{9}{3}+\binom{9}{3}+\binom{12}{3}
\]
blocks, yielding a construction with a total of $388$ blocks.

\section*{Visualization of 10 blocks divided into 3 groups, configuration (2,2,2)}

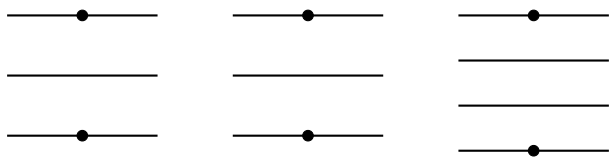
\begin{tikzpicture}[scale=1, every node/.style={scale=1}]

\draw[thick] (0,0) -- (2,0); 
\filldraw[black] (1,0) circle (2pt); 
\draw[thick] (0,-0.8) -- (2,-0.8); 
\draw[thick] (0,-1.6) -- (2,-1.6); 
\filldraw[black] (1,-1.6) circle (2pt); 

\draw[thick] (3,0) -- (5,0); 
\filldraw[black] (4,0) circle (2pt);
\draw[thick] (3,-0.8) -- (5,-0.8); 
\draw[thick] (3,-1.6) -- (5,-1.6); 
\filldraw[black] (4,-1.6) circle (2pt);

\draw[thick] (6,0) -- (8,0); 
\filldraw[black] (7,0) circle (2pt);
\draw[thick] (6,-0.6) -- (8,-0.6); 
\draw[thick] (6,-1.2) -- (8,-1.2); 
\draw[thick] (6,-1.8) -- (8,-1.8); 
\filldraw[black] (7,-1.8) circle (2pt);

\end{tikzpicture}

Assuming that this configuration ocurred, we can take one block at a time from the second group and combine it with pairs taken from the 4-block group.

We do this twice to ensure that there will be 3 numbers (represented by dots) in the first 3-block group, which now becomes a group of 4 blocks.

When we recombine the 12 pairs into triples, we get 
$2 \cdot \binom{12}{3}$ , 440 blocks. Adding these to the previous 388 blocks, we have a total of 828 blocks.

Thus, we can ensure that this construction guarantees:

\begin{equation}\label{eq:triple}
\forall S\in \binom{[60]}{6}\ \exists B\in\mathcal{B}\ \text{such that}\ |S\cap B|\ge 3.
\end{equation}

\section{A $(5-blocks,5-blocks)$ partition guarantees a triple after within-group recombination}
\label{sec:55-guarantee}

Recall the bipartition $[60]=U_1\sqcup U_2$ from Section~\ref{sec:construction}, together with the induced
pair partitions $\mathcal{P}_1$ and $\mathcal{P}_2$ of $U_1$ and $U_2$.
By construction, each $U_t$ is partitioned into $15$ disjoint pairs $\mathcal{P}_t$ ($t\in\{1,2\}$).
Consider the family of recombined $6$-blocks obtained by taking unions of three pairs \emph{within the same group}:
\[
\mathcal{B}_t=\{P\cup P'\cup P'':\ P,P',P''\in\mathcal{P}_t,\ \text{pairwise distinct}\},
\qquad t\in\{1,2\}.
\]

\subsection{Stars-and-bars viewpoint for the $(5,5)$ split}
Let $S\subset[60]$ be any $6$-subset. Set
\[
x_1=|S\cap U_1|,\qquad x_2=|S\cap U_2|.
\]
Since $U_1\cup U_2=[60]$ and the union is disjoint, we have
\[
x_1+x_2=6,\qquad x_1,x_2\in\mathbb{Z}_{\ge 0}.
\]
The solutions of $x_1+x_2=6$ (which represent the number of ways to distribute the $6$ elements of $S$
between the two groups) are:
\[
(0,6),(1,5),(2,4),(3,3),(4,2),(5,1),(6,0).
\]
In every case, at least one of $x_1$ or $x_2$ is at least $3$; equivalently, $S$ contains a triple
fully contained in one of the parts $U_1$ or $U_2$.

\subsection{From a triple to a recombined block}
\begin{lemma}
\label{lem:triple-to-3pairs}
Let $t\in\{1,2\}$ and let $T\subset U_t$ be any $3$-subset.
There exists a block $S'\in\mathcal{B}_t$ such that $T\subset S'$.
\end{lemma}

\begin{proof}
Each element of $U_t$ lies in a unique pair of the partition $\mathcal{P}_t$.
Let $\pi_t(a)$ denote the unique pair in $\mathcal{P}_t$ containing $a\in U_t$.
For $T=\{a,b,c\}$, the set of pairs $\{\pi_t(a),\pi_t(b),\pi_t(c)\}$ has size $2$ or $3$.
If it has size $3$, then $S'=\pi_t(a)\cup \pi_t(b)\cup \pi_t(c)\in\mathcal{B}_t$ and $T\subset S'$.
If it has size $2$, then two of $a,b,c$ lie in the same pair; choose any third pair $P\in\mathcal{P}_t$
distinct from those two pairs and set $S'$ to be the union of these three pairs. In all cases, $T\subset S'$.
\end{proof}

\begin{proposition}[Triple intersection under a $(5,5)$ split]
\label{prop:55-guarantee}
For every $6$-subset $S\subset[60]$ there exists a recombined block $S'\in\mathcal{B}_1\cup\mathcal{B}_2$
such that
\[
|S\cap S'|\ge 3.
\]
\end{proposition}

\begin{proof}
Let $x_1=|S\cap U_1|$ and $x_2=|S\cap U_2|$. Since $x_1+x_2=6$, at least one of them is $\ge 3$.
Assume $x_t\ge 3$ for some $t\in\{1,2\}$. Then $S\cap U_t$ contains a $3$-subset $T$.
By Lemma~\ref{lem:triple-to-3pairs}, there exists $S'\in\mathcal{B}_t$ with $T\subset S'$.
Hence $|S\cap S'|\ge |T|=3$.
\end{proof}

\section{Conclusion}

We gave an explicit family of $910$ blocks of size $6$ guaranteeing a triple intersection with every
$6$-subset of $[60]$.
The construction is elementary and deterministic, and its proof is a short pigeonhole argument.

 We also provided another construction \ref{sec:solutions-x123} that ensures the validity of Equation \ref{eq:triple}. Construction , although it does not guarantee Equation \ref{eq:triple}, but has a probability of success of $\frac{27}{28}.$

The reader may choose any set \( S \subseteq \{1, 2, \dots, 60\} \) containing 6 distinct elements.  
From this set, it is possible to generate all \( \binom{6}{3} = 20 \) combinations of 3 elements (trios).

These 20 trios can then be searched across the 910 blocks generated by the algorithm.  
Due to the construction method of these blocks—which recombines disjoint base blocks through all possible trio pairings within grouped partitions—it is guaranteed that at least some of the trios will appear among the 910 blocks.

This property makes the structure particularly useful for exploring the distribution of number combinations and for performing simulations or statistical analyses and lottery games.

\newpage
\begin{appendices}

\section*{Appendix A:Python Program for Block Generation and Trio Search}

The following Python code was used to generate 910 blocks based on a disjoint partition of the set $\{1, \dots, 60\}$ and search all 3-element combinations (trios) from a given set $S$ of 6 elements.

\begin{lstlisting}[language=Python, caption=Python Code to Generate 910 Blocks and Search Trios]
import random
from itertools import combinations

def split_disjoint_blocks(numbers, block_size=6):
    return [numbers[i:i + block_size] for i in range(0, len(numbers), block_size)]

def get_15_pairs(group):
    pairs = []
    for block in group:
        pairs.extend([tuple(block[i:i + 2]) for i in range(0, 6, 2)])
    return pairs

def generate_blocks_from_pairs(pairs):
    blocks = []
    for p1, p2, p3 in combinations(pairs, 3):
        block = list(p1 + p2 + p3)
        blocks.append(block)
    return blocks

def generate_final_blocks():
    numbers = list(range(1, 61))
    random.shuffle(numbers)
    blocks = split_disjoint_blocks(numbers)
    group_A = blocks[:5]
    group_B = blocks[5:]
    pairs_A = get_15_pairs(group_A)
    pairs_B = get_15_pairs(group_B)
    blocks_A = generate_blocks_from_pairs(pairs_A)
    blocks_B = generate_blocks_from_pairs(pairs_B)
    return blocks_A + blocks_B
\end{lstlisting}

\section*{Appendix B: Example of Application in Lottery Games}

Consider a lottery game where 6 numbers are drawn at random from the set \( \{1, 2, ..., 60\} \).

Using the 910 blocks generated by this construction, we can guarantee that, for any combination of 6 drawn numbers, at least one of the 20 possible 3-number combinations (trios) will appear in one of the 910 tickets.

In other words, by playing all 910 tickets, the player is guaranteed to have at least one trio matching the drawn numbers, and therefore only depends on luck to correctly guess the remaining 3 numbers to complete a full match.

Furthermore, if a player chooses to play only 388 tickets, the probability of obtaining at least one correct trio **in the same ticket as the draw** is \( \frac{27}{28} \approx 96.42\% \). This high probability demonstrates the strong coverage provided by the block design, even with fewer tickets.
\end{appendices}

\end{document}